\documentclass[12pt]{amsart}
\usepackage{ifthen,verbatim}
\usepackage{floatflt,graphicx,graphics}
\usepackage{tikz}
\usepackage{mathrsfs}
\usepackage{amsmath,amsfonts,amssymb,amsthm}
\usepackage{subcaption}

\nonstopmode
\numberwithin{equation}{section}
\theoremstyle{plain}
\newtheorem{theorem}[equation]{Theorem}
\newtheorem{conjecture}[equation]{Conjecture}
\newtheorem{lemma}[equation]{Lemma}
\newtheorem{corollary}[equation]{Corollary}

\newtheorem{proposition}[equation]{Proposition}

\theoremstyle{definition}

\newtheorem{remark}[equation]{Remark}
\newtheorem{nonsec}[equation]{}

\theoremstyle{remark}

\newcommand{\R}{\mathbb{R}}

\newcommand{\B}{\mathbb{B}}

\newcommand{\uhp}{\mathbb{H}}

\newcounter{alphabet}

\newcounter{minutes}
\divide\time by 60
\newcounter{hours}
\multiply\time by 60
\addtocounter{minutes}{-\time}

\begin{document}
\bibliographystyle{amsplain}
\title{Intrinsic metrics defined with arithmetic and logarithmic mean values}

\def\thefootnote{}
\footnotetext{
\texttt{\tiny File:~\jobname .tex,
          printed: \number\year-\number\month-\number\day,
          \thehours.\ifnum\theminutes<10{0}\fi\theminutes}
}
\makeatletter\def\thefootnote{\@arabic\c@footnote}\makeatother

\author[O. Rainio]{Oona Rainio}
\author[R. Kargar]{Rahim Kargar}

\keywords{Hyperbolic geometry, hyperbolic metric, hyperbolic type metrics, intrinsic metrics}
\subjclass[2010]{Primary 51M10; Secondary 51M16}
\begin{abstract}
We introduce several new functions that measure the distance between two points $x$ and $y$ in a domain $G\subsetneq\mathbb{R}^n$ by using the arithmetic or the logarithmic mean of the Euclidean distances from the points $x$ and $y$ to the boundary of $G$. We study in which domains these functions are metrics and find sharp inequalities between them and the hyperbolic metric. We also present one result about their distortion under quasiregular mappings. 
\end{abstract}
\maketitle

\noindent
Oona Rainio$^1$, email: \texttt{ormrai@utu.fi}, ORCID: 0000-0002-7775-7656\newline
Rahim Kargar$^1$, email: \texttt{rakarg@utu.fi}, ORCID: 0000-0003-1029-5386\newline
1: University of Turku, FI-20014 Turku, Finland\\
\textbf{Funding.} O.R.'s research was funded by Finnish Culture Foundation and Magnus Ehrnrooth Foundation and R.K.'s research by Väisälä Foundation.\\
\textbf{Data availability statement.} Not applicable, no new data was generated.\\
\textbf{Conflict of interest statement.} There is no conflict of interest.

\section{Introduction}

The hyperbolic metric is crucial in the study of geometric function theory. Due to its conformal invariance, hyperbolic distances can be easily found in any simply connected planar domains but they are not defined in general in higher dimensions. Because of this, several other generalizations have been introduced for the hyperbolic metric. These intrinsic or hyperbolic type metrics do not share all the properties of the hyperbolic metric, but they have simple definitions regardless of the dimension and, like the hyperbolic metric, they measure the distance between two points in a domain by taking into account how far the points are from each other and what is their position with respect to the boundary of the domain. 

Let $G\subsetneq\R^n$ be a domain. For a point $x\in G$, denote the Euclidean distance to the boundary $\partial G$ by $d_G(x)=\inf_{z\in\partial G}|x-z|$. Let $M$ be a function of the form $M:(0,\infty)\times(0,\infty)\to(0,\infty)$. Now, for a constant $c>0$, define the functions $d^c_G,\widehat{d}^c_G:G\times G\to[0,\infty)$, and $\widetilde{d}^c_G:G\times G\to[0,1)$, as follows: 
\begin{align*}
d^c_G(x,y)&=\frac{|x-y|}{cM(d_G(x),d_G(y)},\\
\widehat{d}^c_G(x,y)&=\log(1+d^c_G(x,y))=\log\left(1+\frac{|x-y|}{cM(d_G(x),d_G(y))}\right),\\
\widetilde{d}^c_G(x,y)&={\rm th}\frac{\widehat{d}^c_2(x,y)}{2}=\frac{|x-y|}{|x-y|+2cM(d_G(x),d_G(y))}.
\end{align*}

There are several intrinsic or hyperbolic type metrics defined as the functions $d^c_G$, $\widehat{d}^c_G$, and $\widetilde{d}^c_G$ for some function $M$. For instance, if $M(x,y)=\min\{x,y\}$, $\widehat{d}^1_G$ is the distance ratio metric introduced by Gehring and Osgood \cite{GO79} and $\widetilde{d}^1_G$ is the $j^*$-metric introduced by Hariri, Vuorinen, and Zhang \cite{hvz}. If $M(x,y)=\sqrt{xy}$, then $\widehat{d}^{1/c}_G$ is the metric $h_{G,c}$ introduced by Dovgoshey, Hariri, and Vuorinen \cite{d16}, assuming $h_{G,c}$ is defined by using the Euclidean distances. If $M(x,y)=(x+y)/2$, $\widetilde{d}^1_G$ is the $t$-metric introduced by Rainio and Vuorinen \cite{inm}. If $M(x,y)=\max\{x,y\}$, then $d^1_G$ is the $\tilde{c}$-metric introduced by Song and Wang \cite{s23}.

The functions $d^c_G$, $\widehat{d}^c_G$, and $\widetilde{d}^c_G$ are not metrics with certain functions $M$ in some domains $G\subsetneq\R^n$, but the next result follows from a well-known result related to metrics. 

\begin{theorem}\label{thm_1}
Let $d^c_G$, $\widehat{d}^c_G$, and $\widetilde{d}^c_G$ be defined with the same function $M$. If $\widehat{d}^c_G$ is a metric in a domain $G\subsetneq\R^n$ for some values of $c$, then so is $\widetilde{d}^c_G$ for the same choices of $G$ and $c$. Similarly, if $d^c_G$ is a metric, then so are both $\widehat{d}^c_G$ and $\widetilde{d}^c_G$.   
\end{theorem}

In this paper, we focus on studying the metrics defined by using the arithmetic mean and its modification as well as the logarithmic mean. In other words, we consider the functions $d^c_G$, $\widehat{d}^c_G$, and $\widetilde{d}^c_G$ defined with
\begin{align*}
M(x,y)=\frac{x+y}{2},\,
M(x,y)=\left(\frac{x^d+y^d}{2}\right)^{1/d},\text{ or }
M(x,y)=
\begin{cases}
\dfrac{x-y}{\log x-\log y},&\text{ if }x\neq y,\\
x,&\text{ if }x=y,
\end{cases}
\end{align*}
where $d>0$ is a constant. We study in which domains the functions $d^c_G$, $\widehat{d}^c_G$, and $\widetilde{d}^c_G$ are metrics and also find several sharp inequalities between $\widetilde{d}^c_G$ and the hyperbolic metric.  

The structure of this paper is as follows. In Section 3, we concentrate on the functions defined with the arithmetic mean and its modification. In Section 4, we consider the case of the logarithmic mean. In Section 5, we find several inequalities between the functions introduced in Sections 3 and 4 and the hyperbolic metric. At the end of this paper, we also present one modification of the Schwarz lemma to describe the distortion of the distances in these new functions under quasiregular mappings.


\section{Preliminaries}

First, we will introduce our notations and define other metrics needed. Let $d_G(x)=\inf_{z\in\partial G}|x-z|$ as in Introduction. Denote the unit vectors of the real space $\R^n$ by $e_1,...,e_n$. For two distinct points $x,y\in\R^n$, denote the Euclidean line segment between these two points by $[x,y]$. 

Denote the upper half-space by $\uhp^n=\{x=(x_1,...,x_n)\in\R^n\,:\,x_n>0\}$ and the unit ball by $\B^n=\{x\in\R^n\,:\,|x|<1\}$. Define the hyperbolic metric as \cite[(4.8), p. 52 \& (4.14), p. 55]{hkv} as
\begin{align*}
\text{ch}\rho_{\uhp^n}(x,y)&=1+\frac{|x-y|^2}{2d_{\uhp^n}(x)d_{\uhp^n}(y)},\quad x,y\in\uhp^n,\\
\text{sh}^2\frac{\rho_{\B^n}(x,y)}{2}&=\frac{|x-y|^2}{(1-|x|^2)(1-|y|^2)},\quad x,y\in\B^n.
\end{align*}
If $n=2$, these formulas can be simplified to
\begin{align*}
\text{th}\frac{\rho_{\uhp^2}(x,y)}{2}=\left|\frac{x-y}{x-\overline{y}}\right|,\quad
\text{th}\frac{\rho_{\B^2}(x,y)}{2}=\left|\frac{x-y}{1-x\overline{y}}\right|,
\end{align*}
where $\overline{y}$ is the complex conjugate of $y$.

The triangular ratio metric, originally introduced by P. H\"ast\"o in 2002 \cite{h} and recently studied in \cite{d23,sch,fss,sinb}, is defined as $s_G:G\times G\to[0,1],$ \cite[(1.1), p. 683]{chkv} 
\begin{align*}
s_G(x,y)=\frac{|x-y|}{\inf_{z\in\partial G}(|x-z|+|z-y|)}. 
\end{align*}

\begin{proposition}\label{prop_dfunc}\cite[Ex. 5.24(1)-(2), p. 80]{hkv}
Let $d$ be a metric in the metric space $X$. Suppose that $f:[0,\infty)\to[0,\infty)$ is an increasing function with $f(0)=0$ so that $f(k)/k$ is decreasing on $(0,\infty)$. Then the function $d'=f\circ d$ is a metric in $X$, too.    
\end{proposition}

\begin{nonsec}{\bf Proof of Theorem \ref{thm_1}.}
Theorem \ref{thm_1} follows directly from Proposition \ref{prop_dfunc} as the functions $f_0,f_1:[0,\infty)\to[0,\infty)$, $f_0(k)=\log(1+k)$ and $f_1(k)={\rm th}(k/2)$ fulfill the conditions of Proposition \ref{prop_dfunc}.   
\end{nonsec}

\section{Arithmetic mean and its modification}


For a domain $G\subsetneq\R^n$ and a constant $0<c\leq1$, define the functions $\phi^c_G,\widehat{\phi}^c_G:G\times G\to[0,\infty)$, $\widetilde{\phi}^c_G:G\times G\to[0,\infty)$ as
\begin{align}
\phi^c_G(x,y)&=\frac{2|x-y|}{c(d_G(x)+d_G(y))},\nonumber\\
\widehat{\phi}^c_G(x,y)&=\log\left(1+\frac{|x-y|}{c(d_G(x)+d_G(y))/2}\right),\nonumber\\
\widetilde{\phi}^c_G(x,y)&=\frac{|x-y|}{|x-y|+c(d_G(x)+d_G(y))}\label{phi}.
\end{align}    
If $c=1$ here, the function $\widetilde{\phi}^c_G$ is equal to the $t$-metric introduced in \cite{inm}.

\begin{theorem}\label{thm_tcg}
For a domain $G\subsetneq\R^n$ and a constant $0<c\leq1$, the function $\widetilde{\phi}^c_G$ is a metric.   
\end{theorem}
\begin{proof}
We only need to show that $\widetilde{\phi}^c_G$ fulfills the triangle inequality since it fulfills the other properties of a metric trivially. By the triangle inequality of the Euclidean metric, we have
\begin{align*}
cd_G(z)\leq c(|z-y|+d_G(y))\leq |z-y|+cd_G(y)    
\end{align*}
for all $y,z\in G$ and $c\in(0,1]$. It follows from this that
\begin{align*}
\widetilde{\phi}^c_G(x,y)
&\leq\frac{|x-z|+|z-y|}{|x-z|+|z-y|+c(d_G(x)+d_G(y))}\\
&=\frac{|x-z|}{|x-z|+|z-y|+cd_G(x)+cd_G(y)}+\frac{|z-y|}{|x-z|+|z-y|+d_G(x)+d_G(y)}\\
&\leq\frac{|x-z|}{|x-z|+c(d_G(x)+d_G(z))}+\frac{|z-y|}{|z-y|+c(d_G(z)+d_G(y))}\\
&=\widetilde{\phi}^c_G(x,z)+\widetilde{\phi}^c_G(z,y),
\end{align*}
which concludes our proof.
\end{proof}

According to Theorem \ref{thm_1}, Theorem \ref{thm_tcg} above would follow if $\widehat{\phi}^c_G$ were a metric in all domains $G\subsetneq\R^n$ for $0<c\leq2$, so it is natural to ask whether the function $\widehat{\phi}^c_G$ is a metric, too. However, this is not the case in all domains, as proven below in Lemmas \ref{lem_TcounterH} and \ref{lem_Tcounterex}.

\begin{lemma}\label{lem_TcounterH}
If the boundary $\partial G$ of a domain $G\subsetneq\R^n$  contains a Euclidean line segment, then the function $\widehat{\phi}^c_G$ is not a metric in $G$ for any values of $c>0$.
\end{lemma}
\begin{proof}
Denote the endpoints of the Euclidean line segment on the boundary by $u$ and $v$, and choose some $q\in[u,v]$. For a small enough $r>0$, fix $k\in G$ so that $|q-k|=r$, $[q,k]$ is perpendicular against $[u,v]$, and the half-ball $\{x\in\R^n\,:\,|x-q|<r,|x-k|<|x-(2q-k)|\}$ is included in $G$. Fix then
\begin{align*}
x=q+\frac{r(u-q)}{4|q-u|}+\frac{h(k-q)}{4},\quad
y=q+\frac{r(q-u)}{4|q-u|}+\frac{h(k-q)}{4},\quad
z=q+\frac{k-q}{4}
\end{align*}
for $0<h<1/4$. Now, for $h\to0^+$, we have
\begin{align*}
\widehat{\phi}^c_G(x,z)+\widehat{\phi}^c_G(z,y)-\widehat{\phi}^c_G(x,y)=
2\log\left(1+\frac{\sqrt{2-2h+h^2}}{c(h+1)}\right)-\log\left(1+\frac{1}{ch}\right)<0.
\end{align*}
It follows that the function $\widehat{\phi}^c_G(x,y)$ is not a metric in this type of a domain.
\end{proof}

\begin{lemma}\label{lem_Tcounterex}
If there is an open ball inside the domain $G\subsetneq\R^n$ so that the endpoints of at least one diameter of this ball are on the boundary $\partial G$, then the function $\widehat{\phi}^c_G(x,y)$ is not a metric for any values of $c>0$.
\end{lemma}
\begin{proof}
Denote the center of the aforementioned ball by $z$. For $k\in(0,1)$, let $x=z+kv$ and $y=z-kv$, where $v\in\R^n$ so that $|v|=d_G(z)$ and $z-v,z+v\in\partial G$. Now, if $k\to1^-$, we have
\begin{align*}
\widehat{\phi}^c_G(x,z)+\widehat{\phi}^c_G(z,y)-\widehat{\phi}^c_G(x,y)=
2\log\left(1+\frac{k}{c(2-k)}\right)-\log\left(1+\frac{k}{c(1-k)}\right)<0
\end{align*}
for all values of $c>0$.
\end{proof}

\begin{remark}
The condition of Lemma \ref{lem_Tcounterex} about  the diameter with endpoints on $\partial G$ holds if the domain $G$ is, for instance, the unit ball, an $n$-dimensional interval, or a twice-punctured space.    
\end{remark}

It follows from Theorem \ref{thm_1} that $\phi^c_G$ cannot be a metric in the domains of Lemmas \ref{lem_TcounterH} and \ref{lem_Tcounterex}. However, both functions $\phi^c_G$ and $\widehat{\phi}^c_G(x,y)$ are metrics when the domain $G$ is the punctured real space $\R^n\setminus\{0\}$. 

\begin{lemma}\label{lem_TcR0}
The functions $\phi^c_G$, $\widehat{\phi}^c_G$ and $\widetilde{\phi}^c_G$ are metrics in $G=\R^n\setminus\{0\}$ for all values of $c>0$.  \end{lemma}
\begin{proof}
The triangular ratio metric $s_G$ is a metric in any domain $G\subsetneq\R^n$. In the punctured real space $\R^n\setminus\{0\}$, it is defined as
\begin{align*}
s_{\R^n\setminus\{0\}}(x,y)=\frac{|x-y|}{|x|+|y|},\quad x,y\in\R^n\setminus\{0\}.    
\end{align*}
Because of the equality $\phi^c_{\R^n\setminus\{0\}}(x,y)=2s_{\R^n\setminus\{0\}}(x,y)/c$, $\phi^c_G$ is a metric in $\R^n\setminus\{0\}$ for all $c>0$ and, by Theorem \ref{thm_1}, so are $\widehat{\phi}^c_G$ and $\widetilde{\phi}^c_G$.
\end{proof}

Let us now consider the modification of the arithmetic mean based on using exponents. Define the functions 
$\phi^{c,d}_G,\widehat{\phi}^{c,d}_G:G\times G\to[0,\infty)$, $\widetilde{\phi}^{c,d}_G:G\times G\to[0,\infty)$ as
\begin{align*}
\phi^{c,d}_G(x,y)&=\frac{|x-y|}{c((d_G(x)^d+d_G(y)^d)/2)^{1/d}},\\
\widehat{\phi}^{c,d}_G(x,y)&=\log\left(1+\frac{|x-y|}{c((d_G(x)^d+d_G(y)^d)/2)^{1/d}}\right),\\
\widetilde{\phi}^{c,d}_G(x,y)&=\frac{|x-y|}{|x-y|+2c((d_G(x)^d+d_G(y)^d)/2)^{1/d}}.
\end{align*}
for a domain $G\subsetneq\R^n$ and constants $c,d>0$.

\begin{theorem}\label{thm_hxy}
For $c,d>0$, the function $\phi^{c,d}_G$ is a metric in the domain $\R^n\setminus\{0\}$ if and only if $d\geq1/3$.
\end{theorem}
\begin{proof}
Follows from \cite[Thm 1.1, p. 39]{h}.
\end{proof}

\begin{corollary}
For $c>0$ and $d\geq1/3$, the functions $\widehat{\phi}^{c,d}_G(x,y)$ and $\widetilde{\phi}^{c,d}_G(x,y)$ are metrics in $G=\R^n\setminus\{0\}$. \end{corollary}
\begin{proof}
Follows from Theorems \ref{thm_1} and \ref{thm_hxy}.
\end{proof}

Based on the computer tests, only the function $\widetilde{\phi}^{c,d}$ is a metric in $G\in\{\uhp^n,\B^n\}$.

\begin{conjecture}
For all $c,d>0$ and $G\in\{\uhp^n,\B^n\}$, the function $\widetilde{\varphi}^c_G$ is a metric but neither $\phi^{c,d}_G$ nor $\widehat{\phi}^{c,d}_G$ fulfills the triangle inequality.    
\end{conjecture}

\section{Logarithmic mean}

For two positive real numbers $x,y>0$, let $L(x,y)$ denote their logarithmic mean value:
\begin{align}
L(x,y)=\frac{x-y}{\log(x)-\log(y)}\quad\text{for}\quad x\neq y,\quad L(x,x)=x.
\end{align}
Let $G\subsetneq\R^n$ be a domain and $x,y$ points in $G$. For $c>0$, we can define the functions $\varphi^c_G,\widehat{\varphi}^c_G:G\times G\to[0,\infty)$, $\widetilde{\varphi}^c_G:G\times G\to[0,1)$ as
\begin{align}
\varphi^c_G(x,y)&=\frac{|x-y|}{cL(d_G(x),d_G(y))},\nonumber\\
\widehat{\varphi}^c_G(x,y)&=\log\left(1+\frac{|x-y|}{cL(d_G(x),d_G(y))}\right),\nonumber\\
\widetilde{\varphi}^c_G(x,y)&=\frac{|x-y|}{|x-y|+2cL(d_G(x),d_G(y))}\label{varphi}.
\end{align}

\begin{lemma}
For all $c>0$, the functions $\varphi^c_G$, $\widehat{\varphi}^c_G$, and $\widetilde{\varphi}^c_G$ are metrics when their domain $G$ is the punctured real space $\R^n\setminus\{0\}$.   
\end{lemma}
\begin{proof}
By \cite[Lemma 3.1(1), p. 46]{h}, $\varphi^1_{\R^n\setminus\{0\}}$ is a metric so trivially $\varphi^c_{\R^n\setminus\{0\}}$ is a metric for all $c>0$ and, by Theorem \ref{thm_1}, so are $\widehat{\varphi}^c_{\R^n\setminus\{0\}}$ and $\widetilde{\varphi}^c_{\R^n\setminus\{0\}}$.   
\end{proof}

\begin{lemma}
Neither $\varphi^c_G$ nor $\widehat{\varphi}^c_G$ is a metric for any $c>0$ when their domain $G$ is the upper half-space $\uhp^n$.    
\end{lemma}
\begin{proof}
Let $x=k+(1-k)i$, $y=-k+(1-k)i$, and $z=e_n$ with $0<k<1$. For all $c>0$, 
\begin{align*}
\widehat{\varphi}^c_{\uhp^n}(x,z)+\widehat{\varphi}^c_{\uhp^n}(z,y)-\widehat{\varphi}^c_{\uhp^n}(x,y)=
2\log\left(1-\frac{\sqrt{2}\log(1-k)}{c}\right) 
-\log\left(1+\frac{2k}{c(1-k)}\right)<0    
\end{align*}
when $k\to1^-$. Consequently, $\widehat{\varphi}^c_{\uhp^n}$ does not fulfill the triangle inequality and is not a metric for any $c>0$. If $\varphi^c_{\uhp^n}$ was a metric, then it would follow from Theorem \ref{thm_1} that $\widehat{\varphi}^c_{\uhp^n}$ is a metric, too, so $\varphi^c_{\uhp^n}$ cannot be a metric.
\end{proof}

\begin{lemma}
Neither $\varphi^c_G$ nor $\widehat{\varphi}^c_G$ is a metric for any $c>0$ when their domain $G$ is the unit ball $\B^n$.    
\end{lemma}
\begin{proof}
Let $x=ke_1$, $y=-ke_1$, and $z=0$ with $0<k<1$. For all $c>0$, 
\begin{align*}
\widehat{\varphi}^c_{\B^n}(x,z)+\widehat{\varphi}^c_{\B^n}(z,y)-\widehat{\varphi}^c_{\B^n}(x,y)=
2\log\left(1-\frac{\log(1-k)}{c}\right) 
-\log\left(1+\frac{2k}{c(1-k)}\right)<0    
\end{align*}
when $k\to1^-$. The function $\widehat{\varphi}^c_{\B^n}$ is therefore not a metric for any $c>0$ and, because of the contradiction that would follow from Theorem \ref{thm_1} if $\varphi^c_{\B^n}$ was a metric, the function $\varphi^c_{\B^n}$ is not a metric, either.        
\end{proof}

However, computer tests suggest that the following result holds.

\begin{conjecture}
For all $c>0$, the function $\widetilde{\varphi}^c_G$ is a metric when $G\in\{\uhp^n,\B^n\}$.    
\end{conjecture}

\section{Inequalities with the hyperbolic metric}

Next, we study the inequalities between the functions $\widetilde{\phi}^c_G$ and $\widetilde{\varphi}^c_G$ defined in \eqref{phi} and \eqref{varphi}, respectively, and the hyperbolic metric $\rho_G$ when $G=\uhp^n$ or $G=\B^n$.

\begin{remark}\label{r_dim}
In the domain $G\in\{\uhp^n,\B^n\}$, the distance between $x,y\in G$ measured by the hyperbolic metric or any of the mean value metrics introduced in this paper depends only on the distances $|x-y|$, $d_G(x)$, and $d_G(y)$. Given the geometry of the domain $G$, the vectors related to these three Euclidean distances are located in the same two-dimensional plane. Thus, our metrics would have the same values if their domain would be the intersection of the original domain $G$ and a two-dimensional plane that either contains the points $x,y,0$ if $G=\B^n$ or contains the points $x,y$ and is perpendicular against the boundary of $\uhp^n$ if $G=\uhp^n$. It follows that it is enough to consider the special case $n=2$ to study the inequalities between these metrics in any dimension $n\geq2$.  
\end{remark}

\begin{theorem}\label{thm_trh}
For all $c>0$ and $x,y\in\uhp^n$, the inequality
\begin{align*}
\frac{1}{1+c}{\rm th}\frac{\rho_{\uhp^n}(x,y)}{2}\leq\widetilde{\phi}^c_{\uhp^n}(x,y)\leq\max\left\{1,\frac{1}{c}\right\}{\rm th}\frac{\rho_{\uhp^n}(x,y)}{2}  \end{align*}
holds with the best possible constants in terms of $c$.
\end{theorem}
\begin{proof}
By Remark \ref{r_dim}, fix $n=2$. By symmetry, fix ${\rm Im}(x)\leq {\rm Im}(y)$. Since both the metrics $\widetilde{\phi}^c_{\uhp^n}(x,y)$ and $\rho_{\uhp^n}(x,y)$ are invariant under a stretching by a factor $r>0$, a translation in a direction parallel to the real axis, and a reflection over the imaginary axis, we can fix $y=i$ and ${\rm Re}(x)\geq0$ without loss of generality. Let us denote $x=k+hi$ where $k\geq0$ and $0<h\leq1$. We have now
\begin{align}\label{quo_trh}
\frac{\widetilde{\phi}^c_{\uhp^2}(x,y)}{{\rm th}(\rho_{\uhp^2}(x,y)/2)}=\frac{|x-\overline{y}|}{|x-y|+c({\rm Im}(x)+{\rm Im}(y))}=\frac{\sqrt{k^2+(h+1)^2}}{\sqrt{k^2+(h-1)^2}+c(h+1)}.    
\end{align}
By differentiation, 
\begin{align*}
&\frac{\partial}{\partial k}\left(\frac{\sqrt{k^2+(h+1)^2}}{\sqrt{k^2+(h-1)^2}+c(h+1)}\right)\\
&=\frac{k(c(h+1)\sqrt{k^2+(h-1)^2}-4h)}{\sqrt{k^2+(h-1)^2}\sqrt{k^2+(h+1)^2}(\sqrt{k^2+(h-1)^2}+c(h+1))^2}=0\\
&\Leftrightarrow\quad k^2=\frac{16h^2}{c^2(h+1)^2}-(h-1)^2.
\end{align*}
We can solve that, for $h\leq1$,
\begin{align*}
\frac{16h^2}{c^2(h+1)^2}-(h-1)^2\geq0  
\quad\Leftrightarrow\quad
ch^2+4h-c\geq0
\quad\Leftrightarrow\quad
h\geq\frac{-2+\sqrt{4+c^2}}{c}. 
\end{align*}
Here,
\begin{align*}
\frac{-2+\sqrt{4+c^2}}{c}<1
\quad\Leftrightarrow\quad
0<4c.
\end{align*} 

Suppose first that $0<h<(-2+\sqrt{4+c^2})/c$. Now, $16h^2/(c^2(h+1)^2)-(h-1)^2<0$ and the quotient \eqref{quo_trh} is therefore strictly increasing with respect to $k\geq0$. Its minimum is 
\begin{align}\label{quo_trh0}
\frac{h+1}{1-h+c(h+1)}    
\end{align}
at $k=0$ and its supremum is 1 at $k\to\infty$. By differentiation, we see that the quotient \eqref{quo_trh0} is increasing with respect to $h$ and has therefore an infimum $1/(1+c)$ at $h\to0^+$.

Consider then the case where $(-2+\sqrt{4+c^2})/c\leq h\leq1$. Now, the quotient \eqref{quo_trh} is at minimum when $k^2=16h^2/(c^2(h+1)^2)-(h-1)^2$. By inputting this value into the quotient \eqref{quo_trh}, we have
\begin{align}\label{quo_trh1}
2\sqrt{\frac{h}{4h+c^2(h+1)^2}}.    
\end{align}
By differentiation, we see that the quotient \eqref{quo_trh1} is increasing with respect to $h$. Its minimum value is obtained when $h=(-2+\sqrt{4+c^2})/c$ or, equivalently, $c=4h/(1-h^2)$. If these equalities hold, the quotient \eqref{quo_trh1} becomes  
\begin{align*}
\frac{1-h}{1+h}
=\frac{c+2-\sqrt{4+c^2}}{c-2+\sqrt{4+c^2}}. 
\end{align*}
However, we can solve that, for all $c>0$, the quotient above is greater than the infimum $1/(1+c)$ found earlier. Since the minimum \eqref{quo_trh1} was the only stationary point of the quotient \eqref{quo_trh}, the supremum is either at $k\to0^+$ or at $k\to\infty$. If $k\to0^+$, the quotient \eqref{quo_trh} becomes the quotient \eqref{quo_trh0}, which is increasing with respect to $h$ and has a maximum $1/c$ at $h=1$. If $k\to\infty$, the quotient \eqref{quo_trh} approaches 1, as noted already earlier. The theorem follows from the found extreme values and their limit values of the quotient \eqref{quo_trh}.
\end{proof}

\begin{theorem}\label{thm_trb}
For all $c>0$ and $x,y\in\B^n$, the inequality
\begin{align*}
\min\left\{\frac{1}{2c},\frac{1}{1+c}\right\}{\rm th}\frac{\rho_{\B^n}(x,y)}{2}\leq\widetilde{\phi}^c_{\B^n}(x,y)\leq\max\left\{1,\frac{1}{c}\right\}{\rm th}\frac{\rho_{\B^n}(x,y)}{2}  \end{align*}
holds with the best possible constants in terms of $c$.
\end{theorem}
\begin{proof}
By Remark \ref{r_dim}, we can assume that $n=2$. By symmetry, we can assume that $|y|\leq|x|$. Since $\widetilde{\phi}^c_{\B^n}(x,y)$ and $\rho_{\B^n}(x,y)$ are invariant under a rotation around the origin, we can assume that $x\in[0,1)$ without loss of generality. Denote $y=|y|e^{ki}$ for $0\leq k<2\pi$. By the law of cosines, we have
\begin{align}\label{quo_trb}
\frac{\widetilde{\phi}^c_{\B^2}(x,y)}{{\rm th}(\rho_{\B^2}(x,y)/2)}=\frac{|1-x\overline{y}|}{|x-y|+c(2-|x|-|y|)}
=\frac{\sqrt{1+|x|^2|y|^2-2|x||y|\cos(k)}}{\sqrt{|x|^2+|y|^2-2|x||y|\cos(k)}+c(2-|x|-|y|)}.    
\end{align}
By differentiation,
\begin{align*}
&\frac{\partial}{\partial\cos(k)}\left(\frac{\sqrt{1+|x|^2|y|^2-2|x||y|\cos(k)}}{\sqrt{|x|^2+|y|^2-2|x||y|\cos(k)}+c(2-|x|-|y|)}\right)\geq0\\
&\quad\Leftrightarrow\quad(1-|x|^2)(1-|y|^2)-c(2-|x|-|y|)\sqrt{|x|^2+|y|^2-2|x||y|\cos(k)}\geq0\\
&\quad\Leftrightarrow\quad\cos(k)\geq\frac{1}{2|x||y|}\left(|x|^2+|y|^2-\frac{(1-|x|^2)^2(1-|y|^2)^2}{c^2(2-|x|-|y|)^2}\right)
\end{align*}
Consequently, we see that the quotient \eqref{quo_trb} has only one possible stationary point with respect to $\cos(k)$ and it is a minimum. At this stationary point, the value of the quotient \eqref{quo_trb} is
\begin{align}\label{quo_cmin}
\sqrt{\frac{(1-|x|^2)(1-|y|^2)}{(1-|x|^2)(1-|y|^2)+c^2(2-|x|-|y|)^2}}.
\end{align}
However, the quotient \eqref{quo_trb} can obtain this minimum only if the stationary point found by differentiation with respect to $\cos(k)$ is on $[-1,1]$ or, equivalently,
\begin{align}\label{ine_trb}
\frac{(1-|x|^2)(1-|y|^2)}{(|x|+|y|)(2-|x|-|y|)}\leq c\leq\frac{(1-|x|^2)(1-|y|^2)}{(|x|-|y|)(2-|x|-|y|)}.
\end{align}

Let us show that the quotient \eqref{quo_cmin} is always greater than or equal to $\min\{1/(2c),1/(1+c)\}$ when the inequality \eqref{ine_trb} holds. Suppose first that $c\leq1$. This is possible within bounds of \eqref{ine_trb} if and only if
\begin{align}\label{ine_trb0}
\frac{(1-|x|^2)(1-|y|^2)}{(|x|+|y|)(2-|x|-|y|)}\leq1
\quad\Leftrightarrow\quad
(1+|x||y|)^2\leq2(|x|+|y|).
\end{align}
We want to show that
\begin{align*}
\frac{1}{1+c}\leq\sqrt{\frac{(1-|x|^2)(1-|y|^2)}{(1-|x|^2)(1-|y|^2)+c^2(2-|x|-|y|)^2}}
\quad\Leftrightarrow\quad
c\left(1-\frac{(2-|x|-|y|)^2}{(1-|x|^2)(1-|y|^2)}\right)\geq-2.
\end{align*}
Since $(1-|x|^2)(1-|y|^2)<(2-|x|-|y|)^2$, it is enough to prove the latter inequality above holds for the smallest value of $c$ to show it holds for all possible values of $c\leq1$. By inputting the lower bound of $c$ from \eqref{ine_trb}, we have
\begin{align*}
&\frac{(1-|x|^2)(1-|y|^2)}{(|x|+|y|)(2-|x|-|y|)}\left(1-\frac{(2-|x|-|y|)^2}{(1-|x|^2)(1-|y|^2)}\right)\geq-2\\
&\Leftrightarrow\quad
|x||y|\geq1-2|x|-2|y|.
\end{align*}
It can be shown that the inequality above holds for all $|x|,|y|\in[0,1)$ such that $(1+|x||y|)^2\leq2(|x|+|y|)$.

Let us then consider the case $c>1$. By the inequality \eqref{ine_trb}, this is possible only if
\begin{align}
\frac{(1-|x|^2)(1-|y|^2)}{(|x|-|y|)(2-|x|-|y|)}\geq1
\quad\Leftrightarrow\quad
1-2|y|^2+|x|^2|y|^2\geq2(|x|-|y|).
\end{align}
Our inequality is
\begin{align*}
\frac{1}{2c}\leq\sqrt{\frac{(1-|x|^2)(1-|y|^2)}{(1-|x|^2)(1-|y|^2)+c^2(2-|x|-|y|)^2}}
\quad\Leftrightarrow\quad
c^2\left(4-\frac{(2-|x|-|y|)^2}{(1-|x|^2)(1-|y|^2)}\right)\geq1,
\end{align*}
where $4(1-|x|^2)(1-|y|^2)>(2-|x|-|y|)^2$ for all $0\leq|y|\leq|x|<1$ such that $1-2|y|^2+|x|^2|y|^2\geq2(|x|-|y|)$. Consequently, we need to prove this inequality only for the smallest value of $c$. Since we set $c>1$ and the inequality \eqref{ine_trb} gives a lower bound for $c$, the inequality becomes 
\begin{align*}
&\max\left\{1,\frac{(1-|x|^2)(1-|y|^2)}{(|x|+|y|)(2-|x|-|y|)}\right\}^2\left(4-\frac{(2-|x|-|y|)^2}{(1-|x|^2)(1-|y|^2)}\right)\geq1\\
&\Leftrightarrow\quad
\begin{cases}
(2-|x|-|y|)^2\leq3(1-|x|^2)(1-|y|^2),&\quad\text{if \eqref{ine_trb0} holds,}\\
(1+|x||y|)(2-|x|-|y|)\leq4(1-|x|^2)(1-|y|^2),&\quad\text{if \eqref{ine_trb0} does not hold.}
\end{cases}
\end{align*}
Both of these two inequalities above can be proven to hold for all such $0\leq|y|\leq|x|<1$ that satisfy $1-2|y|^2+|x|^2|y|^2\geq2(|x|-|y|)$ and for which the inequality \eqref{ine_trb0} either holds or does not, as specified above. It follows from this that the found minimum of the quotient \eqref{quo_cmin} is never below $\min\{1/(2c),1/(1+c)\}$.

Let us then study the values of the quotient \eqref{quo_trb} when $\cos(k)=-1$ or $\cos(k)=1$. Suppose first that $\cos(k)=-1$. Now, the quotient \eqref{quo_trb} becomes
\begin{align}\label{quo_trb0}
\frac{1+|x||y|}{|x|+|y|+c(2-|x|-|y|)}.    
\end{align}
By differentiation,
\begin{align}
\frac{\partial}{\partial|y|}\left(\frac{1+|x||y|}{|x|+|y|+c(2-|x|-|y|)}\right)=\frac{-1+|x|^2+c(1+2|x|-|x|^2)}{(|x|+|y|+c(2-|x|-|y|))^2}.   
\end{align}
Consequently, the quotient \eqref{quo_trb0} is monotonic with respect to $|y|$ and obtains its extreme values when $|y|=0$ or $|y|=|x|$. At $|y|=0$, the quotient \eqref{quo_trb0} is $1/(|x|+c(2-|x|))$, whose infimum and supremum with respect to $|x|$ are $1/(2c)$ at $|x|\to0^+$ and $1/(1+c)$ at $|x|\to1^-$ or vice versa, depending if $c\leq1$ or not. If $|y|=|x|$,  the quotient \eqref{quo_trb0} becomes
\begin{align}\label{quo_trb01}
\frac{1+|x|^2}{2(|x|+c(1-|x|))}.    
\end{align}
By differentiation,
\begin{align*}
\frac{\partial}{\partial|x|}\left(\frac{1+|x|^2}{2(|x|+c(1-|x|))}\right)=\frac{(1-c)|x|^2-1+2c|x|-1+c}{2(|x|+c(1-|x|))^2}&=0\\
\Leftrightarrow\quad|x|&=\frac{-c\pm\sqrt{2c^2-2c+1}}{1-c}
\end{align*}
The root with a minus sign in place of $\pm$ does not belong in $(0,1)$ for any values of $c$. The other root belongs on $(0,1)$ if and only if $c<1$. Consequently, if $c<1$, the minimum of the quotient \eqref{quo_trb01} is $(\sqrt{2c^2-2c+1}-c)/(1-c)^2$ at $|x|=(-c+\sqrt{2c^2-2c+1})/(1-c)$. However, since for all $0<c<1$, we have
\begin{align*}
\frac{\sqrt{2c^2-2c+1}-c}{(1-c)^2}<\frac{1}{1+c}
\quad\Leftrightarrow\quad
(1-c)^3>0,
\end{align*}
this minimum is never the smallest possible value of the quotient \eqref{quo_trb}. The quotient \eqref{quo_trb01} becomes $1/(2c)$ when $x\to0^+$ and 1 when $x\to1^-$.

If we fix $\cos(k)=1$, the quotient \eqref{quo_trb} becomes
\begin{align}\label{quo_trb1}
\frac{1-|x||y|}{|x|-|y|+c(2-|x|-|y|)}.    
\end{align}
By differentiation,
\begin{align*}
\frac{\partial}{\partial|y|}\left(\frac{1-|x||y|}{|x|-|y|+c(2-|x|-|y|)}\right)=\frac{(1-|x|)(1+|x|+c(1-|x|))}{(|x|-|y|+c(2-|x|-|y|))^2}>0,    
\end{align*}
so the quotient \eqref{quo_trb1} is increasing with respect to $|y|$. It has a minimum $1/(|x|+c(2-|x|)$ at $|y|=0$ and the infimum of this minimum with respect to $|x|$ is either $1/(2c)$ at $x\to0^+$ or $1/(1+c)$ at $x\to1^-$. The maximum of the quotient \eqref{quo_trb1} with respect to $|y|$ is $(1+|x|)/(2c)$ at $|y|=|x|$ and the supremum of this value with respect to $|x|$ is $1/c$ at $|x|\to1^-$.
\end{proof}

\begin{lemma}\label{lem_hit}
For all $c>0$ and points $x,y$ in a domain $G\subsetneq\R^n$, the inequality $\widetilde{\varphi}^c_G(x,y)\geq \widetilde{\phi}^c_G(x,y)$.    
\end{lemma}
\begin{proof}
Follows from the fact that $L(x,y)\leq(x+y)/2$ for all $x,y>0$.    
\end{proof}

\begin{theorem}\label{thm_hit}
For all $c>0$ and $x,y\in\uhp^n$, the inequality
\begin{align*}
\frac{1}{1+c}{\rm th}\frac{\rho_{\uhp^n}(x,y)}{2}\leq\widetilde{\varphi}^c_{\uhp^n}(x,y)\leq\max\left\{1,\frac{1}{c}\right\}{\rm th}\frac{\rho_{\uhp^n}(x,y)}{2}  \end{align*}
holds and $\max\{1,1/c\}$ is the best possible constant here in terms of $c$.
\end{theorem}
\begin{proof}
The first part of the inequality follows directly from Theorem \ref{thm_trh} and Lemma \ref{lem_hit}. To prove the second part of the inequality, we need to find the maximum of the quotient $\widetilde{\varphi}^c_{\uhp^n}(x,y)/{\rm th}(\rho_{\uhp^n}(x,y)/2)$. By using the same reasoning as in the proof of Theorem \ref{thm_trh}, we can assume that $n=2$, $y=i$ and $x=k+hi$ with $k>0$ and $0<h\leq1$ without loss of generality. Now,
\begin{align}\label{quo_phih}
\frac{\widetilde{\varphi}^c_{\uhp^2}(x,y)}{{\rm th}(\rho_{\uhp^2}(x,y)/2)}=\frac{|x-\overline{y}|}{|x-y|+2cL({\rm Im}(x),{\rm Im}(y))}=\frac{\sqrt{k^2+(h+1)^2}}{\sqrt{k^2+(h-1)^2}+2c(1-h)/(-\log h)}.    
\end{align}
By differentiation, 
\begin{align*}
&\frac{\partial}{\partial k}\left(\frac{\sqrt{k^2+(h+1)^2}}{\sqrt{k^2+(h-1)^2}+2c(1-h)/(-\log h)}\right)\\
&=\frac{2k(c(1-h)\sqrt{k^2+(h-1)^2}+2h\log h)}{(-\log h)\sqrt{k^2+(h-1)^2}\sqrt{k^2+(h+1)^2}(\sqrt{k^2+(h-1)^2}+2c(1-h)/(-\log h)^2}=0\\
&\Leftrightarrow\quad k^2=\frac{4h^2\log^2h}{c^2(1-h)^2}-(1-h)^2.
\end{align*}
We see that the quotient \eqref{quo_phih} has only one possible stationary point with respect to $k>0$ and this point is a minimum. Consequently, the maximum is found when $k\to0^+$ or $k\to\infty$. If $k\to\infty$, the quotient \eqref{quo_phih} approaches 1. If $k\to0^+$ instead, the quotient \eqref{quo_phih} becomes
\begin{align*}
\frac{h+1}{(1-h)(1-2c/\log h)}.    
\end{align*}
Let us now prove that the quotient above is less than $\max\{1,1/c\}$. If $c\leq1$, the inequality is
\begin{align}\label{ine_phiq}
\frac{h+1}{(1-h)(1-2c/\log h)}\leq\frac{1}{c}
\quad\Leftrightarrow\quad
c(1+h)+2c(1-h)/\log h-(1-h)\leq0.
\end{align}
Since the expression $c(1+h)+2c(1-h)/\log h-(1-h)$ is clearly monotonic with respect to $0<c\leq1$, the inequality \eqref{ine_phiq} for all possible values of $c\in(0,1]$ follows from the inequalities
\begin{align*}
&\lim_{c\to0^+}(c(1+h)+2c(1-h)/(\log h)-(1-h))=-1+h\leq0,\\
&(1+h)+2(1-h)/(\log h)-(1-h)=2(h\log h-h+1)/\log h\leq0,
\end{align*}
which clearly holds for all $0<h\leq1$. If $c>1$, we have
\begin{align*}
&\frac{h+1}{(1-h)(1-2c/\log h)}\leq1
\quad\Leftrightarrow\quad
2h-2c(1-h)/\log h\leq0\\
&\Rightarrow\quad
2h-2(1-h)/\log h=2(h\log-h+1)/\log h\leq0.
\end{align*}
Because the quotient \eqref{quo_phih} approaches 1 as $k\to\infty$ and
\begin{align*}
\lim_{h\to1^-}\frac{h+1}{(1-h)(1-2c/\log h)}=\frac{1}{c},   \end{align*}
the bound $\max\{1,1/c\}$ is the best possible one in the theorem.
\end{proof}

\begin{theorem}\label{thm_hit1}
For all $c>0$ and $x,y\in\B^n$, the inequality
\begin{align*}
\min\left\{\frac{1}{2c},\frac{1}{1+c}\right\}{\rm th}\frac{\rho_{\B^n}(x,y)}{2}\leq\widetilde{\varphi}^c_{\B^n}(x,y)\leq\max\left\{1,\frac{1}{c}\right\}{\rm th}\frac{\rho_{\B^n}(x,y)}{2}  \end{align*}
holds and $\max\{1,1/c\}$ is the best possible constant here in terms of $c$.
\end{theorem}
\begin{proof}
The first part of the inequality follows from Theorem \ref{thm_trb} and Lemma \ref{lem_hit}. For the second part, we need the maximum of the quotient $\widetilde{\varphi}^c_{\B^n}(x,y)/{\rm th}(\rho_{\B^n}(x,y)/2)$. By using the same reasoning as in the proof of Theorem \ref{thm_trb}, we can assume that $n=2$, $x\in[0,1)$ and $y=|y|e^{ki}$ with $0\leq|y|\leq|x|<1$ and $0\leq k<2\pi$ without loss of generality. By the law of cosines,
\begin{align}\label{quo_phi1}
\frac{\widetilde{\varphi}^c_{\B^2}(x,y)}{{\rm th}(\rho_{\B^2}(x,y)/2)}=\frac{\sqrt{1+|x|^2|y|^2-2|x||y|\cos(k)}}{\sqrt{|x|^2+|y|^2-2|x||y|\cos(k)}+2cL(1-|x|,1-|y|)}.    
\end{align}
By differentiation of the quotient \eqref{quo_phi1} with respect to $\cos(k)$ like in the proof of Theorem \ref{thm_trb}, we can see that it has only one possible stationary point, which is a minimum. Thus, its maximum is found at $\cos(k)=-1$ or at $\cos(k)=1$. Assume below that all the $\pm$ signs in the same formula mean a plus and all the $\mp$ signs mean a minus, or vice versa. The quotient \eqref{quo_phi1} becomes
\begin{align}\label{quo_y1}
\frac{1\mp|x||y|}{|x|\mp|y|+2cL(1-|x|,1-|y|)}\quad\text{at}\quad\cos(k)=\pm1.    
\end{align}
Denote $f(1-|y|)=2cL(1-|x|,1-|y|)$. By differentiation,
\begin{align*}
\frac{\partial}{\partial|y|}\left(\frac{1\mp|x||y|}{|x|\mp|y|+f(1-|y|)}\right)
=\frac{\pm(1-|x|^2)\mp|x|f(1-|y|)+(1\mp|x||y|)f'(1-|y|)}{(|x|\mp|y|+f(1-|y|))^2},\\
\frac{\partial}{\partial|y|}(\pm(1-|x|^2)\mp|x|f(1-|y|)+(1\mp|x||y|)f'(1-|y|))=-(1\pm|x||y|)f''(1-|y|).
\end{align*}
Here, $f''(1-|y|)$ is negative because the first derivative of $2cL(1-|x|,1-|y|)$ with respect to $|y|$ is decreasing on $(0,|x|)$. Thus, the derivative of the quotient \eqref{quo_y1} is increasing and the quotient \eqref{quo_y1} can only have a stationary point that is a minimum, regardless of whether the $\mp$ signs mean a plus or a minus. Consequently, the maximum value of the quotient \eqref{quo_y1} is obtained when either $|y|=0$ or $|y|=|x|$. This gives us three possible quotients:
\begin{align*}
\frac{1}{|x|+2cL(1-|x|,1)},\quad
\frac{1+|x|}{2c},\quad
\frac{1+|x|^2}{2(|x|+c(1-|x|))}.
\end{align*}
By differentiation with respect to $|x|$, we can prove that the possible maximum values of the three quotients above include 1, $1/(2c)$, and $1/c$.
\end{proof}

\section{Distortion under quasiregular mappings}
Let us yet present one result about the distortion of the distances in the functions $\widetilde{\phi}^c_G$ and $\widetilde{\varphi}^c_G$ defined in \eqref{phi} and \eqref{varphi} under quasiregular mappings. See the definition of $K$-quasiregular mappings from \cite{hkv} or \cite{sch}. Define a number $\lambda_n$ by the formula \cite[(9.5) p. 157 \& (9.6), p. 158]{hkv} 
\begin{align*}
\log\lambda_n=\lim_{t\to\infty}((\gamma_n(t)\slash\omega_{n-1})^{1\slash(1-n)}-\log t),   
\end{align*}
where $\gamma_n$ is the Gr\"otzsch capacity as defined in \cite[(7.17), p. 121]{hkv} and $\omega_{n-1}$ is the $(n-1)$-dimensional surface area of the unit sphere. It follows that $4\leq\lambda_n<2e^{n-1}$ for each $n\geq2$ and $\lambda_2=4$.

\begin{theorem}\label{thm_sch}\cite[Thm 16.2, p. 300]{hkv}
If $f:G\to G'$ with $G,G'\in\{\uhp^n,\B^n\}$ is a non-constant $K$-quasiregular mapping and $\alpha=K^{1/(1-n)}$,
then for all $x,y\in G$,
\begin{align*}
&(1)\quad{\rm th}(\rho_{G'}(f(x),f(y))/2)
\leq\lambda_n^{1-\alpha}({\rm th}(\rho_G(x,y)/2))^\alpha,\\
&(2)\quad{\rm th}(\rho_{G'}(f(x),f(y))/2)
\leq K({\rm th}\rho_G(x,y)+\log4).
\end{align*}
\end{theorem}

The Schwarz lemma can be written for several intrinsic metrics and quasi-metrics, as in the following theorem.

\begin{corollary}
If $f:\uhp^n\to\uhp^n$ is a non-constant $K$-quasiregular mapping and $\alpha=K^{1/(1-n)}$, then the inequalities 
\begin{align*}
&(1)\quad\widetilde{\phi}^c_{\uhp^n}(f(x),f(y))
\leq\lambda_n^{1-\alpha}\max\left\{1,\frac{1}{c}\right\}\left((1+c)\widetilde{\phi}^c_{\uhp^n}(x,y)\right)^\alpha,\\    
&(2)\quad\widetilde{\varphi}^c_{\uhp^n}(f(x),f(y))
\leq\lambda_n^{1-\alpha}\max\left\{1,\frac{1}{c}\right\}\left((1+c)\widetilde{\varphi}^c_{\uhp^n}(x,y)\right)^\alpha
\intertext{hold for all $x,y\in\uhp^n$. Furthermore, for a $K$-quasiregular mapping $f:\B^n\to\B^n$,}
&(3)\quad\widetilde{\phi}^c_{\B^n}(f(x),f(y))
\leq\lambda_n^{1-\alpha}\max\left\{1,\frac{1}{c}\right\}\left(\max\left\{2c,1+c\right\}\widetilde{\phi}^c_{\B^n}(x,y)\right)^\alpha,\\
&(4)\quad\widetilde{\varphi}^c_{\B^n}(f(x),f(y))\leq\lambda_n^{1-\alpha}\max\left\{1,\frac{1}{c}\right\}\left(\max\left\{2c,1+c\right\}\widetilde{\varphi}^c_{\B^n}(x,y)\right)^\alpha. \end{align*}
\end{corollary}
\begin{proof}
Follows from Theorems \ref{thm_trh}, \ref{thm_trb}, \ref{thm_hit}, \ref{thm_hit1}, and \ref{thm_sch}.     
\end{proof}


\def\cprime{$'$} \def\cprime{$'$} \def\cprime{$'$}
\providecommand{\bysame}{\leavevmode\hbox to3em{\hrulefill}\thinspace}
\providecommand{\MR}{\relax\ifhmode\unskip\space\fi MR }
\providecommand{\MRhref}[2]{%
  \href{http://www.ams.org/mathscinet-getitem?mr=#1}{#2}
}
\providecommand{\href}[2]{#2}


\end{document}